\documentclass[12pt, reqno]{amsart}

\usepackage[a4paper, centering, total={170mm,240mm}]{geometry}

\usepackage{amssymb,latexsym}
\usepackage{blindtext}
\usepackage[backgroundcolor=olive, linecolor=olive, textsize=tiny]{todonotes}
\usepackage{esint,dsfont}

\usepackage{palatino}

\usepackage[colorlinks=true, linkcolor=blue, citecolor=purple, urlcolor=blue]{hyperref}

\newtheorem{prop}{\bf Proposition}[section]
\newtheorem{thm}[prop]{\bf Theorem}
\newtheorem{cor}[prop]{\bf Corollary}
\newtheorem{lem}[prop]{\bf Lemma}
\newtheorem{rmk}[prop]{\it Remark}

\newtheorem{conj}{\bf Conjecture}

\begin{document}

\title[CND kernels and proper actions on $E\subseteq L^1$]{{\bf\Large Almost invariant CND kernels and proper uniformly Lipschitz actions on subspaces of $L^1$}}

\author[I. Vergara]{Ignacio Vergara}
\address{Departamento de Matem\'atica y Ciencia de la Computaci\'on, Universidad de Santiago de Chile, Las Sophoras 173, Estaci\'on Central, Chile.}

\email{ign.vergara.s@gmail.com}
\thanks{Part of the research was carried out at the Leonhard Euler St. Petersburg International Mathematical Institute and supported by the Ministry of Science and Higher Education of the Russian Federation (Agreement № 075–15–2022–287 dated 06.04.2022).}

\makeatletter
\@namedef{subjclassname@2020}{%
  \textup{2020} Mathematics Subject Classification}
\makeatother

\subjclass[2020]{Primary 22D55; Secondary 43A35, 20F67, 46E30}
%

\keywords{Uniformly Lipschitz affine actions, subspaces of $L^1$, conditionally negative definite kernels, products of quasi-trees, weakly amenable groups, a-TTT-menable groups}

\begin{abstract}
We define the notion of almost invariant conditionally negative definite kernel and use it to give a characterisation of groups admitting a proper uniformly Lipschitz affine action on a subspace of an $L^1$ space. We show that this condition is satisfied by groups acting properly on products of quasi-trees, weakly amenable groups with Cowling--Haagerup constant 1, and a-TTT-menable groups.
\end{abstract}


\begingroup
\def\uppercasenonmath#1{} 
\let\MakeUppercase\relax 
\maketitle
\endgroup

\section{{\bf Introduction}}

A group is said to have the \textit{Haagerup property} (or that it is \textit{a-T-menable}) if it admits a (metrically) proper isometric affine action on a Hilbert space.

Every Hilbert space can be isometrically embedded into an $L^1$ space; see for instance \cite[Theorem 6.8]{ChDrHa}. On the other hand, the fact that the norm on an $L^1$ space defines a conditionally negative definite kernel shows that every group acting properly by isometries on a subspace of an $L^1$ space has the Haagerup property; see for instance \cite[Corollary 6.23]{ChDrHa}. As a consequence, the Haagerup property can be characterised by the existence of a proper isometric affine action on a subspace of an $L^1$ space.

In this paper, we study a weakening of this property by considering actions that are not necessarily isometric, but which are uniformly Lipschitz.

Let $\Gamma$ be a countable group and let $E$ be a (real) Banach space. We denote by $\mathcal{B}(E)$ the algebra of bounded operators on $E$, and by $\operatorname{GL}(E)$ the subset of $\mathcal{B}(E)$ consisting of invertible operators. Observe that $\operatorname{GL}(E)$ is a group for the composition. A \textit{representation} of $\Gamma$ on $E$ is a group homomorphism $\pi:\Gamma\to\operatorname{GL}(E)$. An \textit{affine action} $\Gamma\curvearrowright^\sigma E$ is given by a representation $\pi:\Gamma\to\operatorname{GL}(E)$ and a map $b:\Gamma\to E$ satisfying
\begin{align*}
\forall s,t\in\Gamma,\quad b(st)=\pi(s)b(t)+b(s).
\end{align*}
We say that $b$ is a \textit{cocycle} for the representation $\pi$. The aforementioned affine action is then given by
\begin{align}\label{affn_act}
\forall s\in\Gamma,\ \forall v\in E,\quad \sigma(s)v=\pi(s)v+b(s).
\end{align}
We call $\pi$ the \textit{linear part} of $\sigma$. We say that the action $\sigma$ is $C$\textit{-Lipschitz} ($C\geq 1$) if
\begin{align*}
\forall s\in\Gamma,\ \forall v,w\in E,\quad \lVert \sigma(s)v-\sigma(s)w\rVert\leq C\lVert v-w\rVert.
\end{align*}
This is equivalent to the fact that $\pi$ is \textit{uniformly bounded} by $C$:
\begin{align*}
\sup_{s\in\Gamma}\lVert \pi(s)\rVert \leq C.
\end{align*}
We say that $\sigma$ \textit{uniformly Lipschitz} if it is $C$-Lipschitz for some $C\geq 1$. Notice that an action is $1$-Lipschitz if and only if it is isometric.

We say that an action of a group $\Gamma$ on a metric space $(X,d)$ is \textit{(metrically) proper} if, for every $x\in X$ and every $R>0$, the set
\begin{align*}
\{s\in\Gamma\ \mid\ d(s\cdot x,x)\leq R\}
\end{align*}
is finite. We usually write this condition as
\begin{align*}
d(s\cdot x,x)\to\infty\quad\text{as}\quad s\to\infty.
\end{align*}
Observe that a uniformly Lipschitz affine action as in \eqref{affn_act} is proper if and only if the cocycle $b$ is proper:
\begin{align*}
\lVert b(s)\rVert \to\infty\quad\text{as}\quad s\to\infty.
\end{align*}

Our first result gives a necessary and sufficient condition for a group to admit a proper uniformly Lipschitz affine action on a subspace of an $L^1$ space. Let $\Gamma$ be a countable group and let $\psi:\Gamma\times\Gamma\to[0,\infty)$ be a conditionally negative definite (CND) kernel. This means that there is a Hilbert space $\mathcal{H}$ and a map $f:\Gamma\to\mathcal{H}$ such that
\begin{align}\label{def_CND}
\psi(x,y)=\lVert f(x)-f(y)\rVert^2,\quad\forall x,y\in\Gamma.
\end{align}
We say that $\psi$ is \textit{almost invariant} if there exists $C>0$ such that
\begin{align*}
\psi(sx,sy)\leq\psi(x,y) + C,\quad\forall s,x,y\in\Gamma.
\end{align*}
In addition, we say that $\psi$ is \textit{proper} if
\begin{align*}
\psi(x,e)\to\infty\quad\text{as}\quad x\to\infty,
\end{align*}
where $e$ is the identity element of $\Gamma$.

\begin{thm}\label{Thm_caract}
Let $\Gamma$ be a countable group. The following are equivalent:
\begin{itemize}
\item[(i)] There exists a proper almost invariant CND kernel $\psi:\Gamma\times\Gamma\to[0,\infty)$.
\item[(ii)] There exist a measure space $(\Omega,\mu)$, a closed subspace $E\subseteq L^1(\Omega,\mu)$, and a proper uniformly Lipschitz affine action of $\Gamma$ on $E$.
\item[(ii)'] For every $\varepsilon>0$, there exist a measure space $(\Omega,\mu)$, a closed subspace $E\subseteq L^1(\Omega,\mu)$, and a proper $(1+\varepsilon)$-Lipschitz affine action of $\Gamma$ on $E$.
\end{itemize}
\end{thm}

In the case of isometric actions, a group acts properly on a Hilbert space if and only if it does so on a subset of an $L^p$ space with $p\in(0,2]$; see \cite[Corollary 1.5]{ChDrHa} and \cite[Corollary 6.23]{ChDrHa}. We do not know if this characterisation can be extended to uniformly Lipschitz actions, but Theorem \ref{Thm_caract} allows us to conclude the following.

\begin{cor}\label{Cor_Lp}
Let $p\in(0,2]$ and let $\Gamma$ be a countable group admitting a proper uniformly Lipschitz action on a subset of an $L^p$ space. Then $\Gamma$ has a proper uniformly Lipschitz affine action on a subspace of an $L^1$ space.
\end{cor}

One nice feature of Theorem \ref{Thm_caract} is that almost invariant CND kernels arise in various contexts, which allows us to construct uniformly Lipschitz actions on $E\subseteq L^1$ for different classes of groups. We begin by focusing on groups acting on products of quasi-trees. A quasi-tree is a graph that is quasi-isometric to a tree. Here we view (the set of vertices of) a graph as a metric space endowed with the edge-path distance. For a finite family of graphs $(X_i,d_i)$ $(i=1,\ldots,N)$, we endow the product $X=X_1\times\cdots\times X_N$ with the $\ell^1$-combination of the metrics:
\begin{align}\label{l1_combi}
d(x,y)=\sum_{i=1}^N d_i(x_i,y_i).
\end{align}
This is exactly the edge-path distance on $X$ with its natural graph structure.

\begin{thm}\label{Thm_QT}
Let $\Gamma$ be a countable group acting properly by isometries on a finite product of quasi-trees. Then $\Gamma$ admits a proper uniformly Lipschitz affine action on a subspace of an $L^1$ space.
\end{thm}

The class of groups admitting proper actions on products of quasi-trees is very rich. It contains mapping class groups \cite{BeBrFu}, residually finite hyperbolic groups \cite{BeBrFu}, and some $3$-manifold groups \cite{HaNgYa}; see also \cite{NguYan} for more examples. Theorem \ref{Thm_QT} is motivated by the following conjecture; see \cite[Conjecture 35]{Now2}.

\begin{conj}[Shalom]\label{conj_Sha}
Every hyperbolic group admits a proper uniformly Lipschitz affine action on a Hilbert space.
\end{conj}

This conjecture has been verified for lattices in the rank 1 Lie groups $\operatorname{Sp}(n,1)$ and $F_{4,-20}$; see \cite[Theorem A]{Nis} and \cite[Corollary 1.7]{Ver2}. These are very interesting classes of groups because they satisfy Property (T), meaning that every isometric action on a Hilbert space has a fixed point, which implies in particular that such actions cannot be proper. As a matter of fact, ``most'' hyperbolic groups have Property (T); see  \cite{KotKot, Zuk} for precise statements and proofs. However, for uniformly Lipschitz actions, the situation is not so well understood, and Conjecture \ref{conj_Sha} remains open in general.

In light of the results in \cite{BeBrFu}, Theorem \ref{Thm_QT} gives a solution to a variation of Conjecture \ref{conj_Sha} for residually finite hyperbolic groups, by replacing $L^2$ by a subspace of $L^1$. Furthermore, the residual finiteness hypothesis can be removed. More precisely, every hyperbolic group admits a proper uniformly Lipschitz affine action on $E\subseteq L^1$; see \cite[Theorem 1.2]{Ver} or Theorem \ref{Thm_a-TTT} below for a more general statement. The proof of this result also consists in constructing a proper almost invariant CND kernel.

We will give two different proofs of Theorem \ref{Thm_QT}. Our original proof is inspired by \cite{CDHL}, and it consists in showing that a certain approximation of the edge-path distance on a quasi-tree is a CND kernel. Afterwards, we realised that the main result of \cite{Ker} can be used to give a much shorter proof by simply considering the pull-back of the metric on a tree. We believe that both proofs are interesting, and that they provide different insights on the existence of such kernels. Hence we decided to present them both in Section \ref{S_QT}.

Let us mention some more recent progress in this direction. Using different methods, Dru\c{t}u and Mackay \cite{DruMac} showed that mapping class groups and residually finite hyperbolic groups admit proper $(2+\varepsilon)$-Lipschitz actions on $\ell^1$ and on $L^1([0,1])$. As observed in \cite[Remark 1.12]{DruMac}, for groups with Property (T), this cannot be improved in order to obtain $(1+\varepsilon)$-Lipschitz actions as in Theorem \ref{Thm_caract}, which shows that the class of subspaces of $L^1$ provides a much less rigid framework than the space $L^1$ itself.

Next, we apply Theorem \ref{Thm_caract} to weakly amenable groups. A countable group $\Gamma$ is said to be \textit{weakly amenable} if there exists a sequence of finitely supported functions $\varphi_n:\Gamma\to\mathbb{C}$ converging pointwise to 1, and a constant $C\geq 1$ such that $\sup_n\lVert \varphi_n\rVert _{B_2(\Gamma)}\leq C$. Here $B_2(\Gamma)$ stands for the space of Herz--Schur multipliers on $\Gamma$; see \cite[Appendix D]{BroOza} for details. The \textit{Cowling--Haagerup constant} $\boldsymbol\Lambda(\Gamma)$ is the infimum of all $C\geq 1$ such that the condition above holds. Our second result is also motivated by an open conjecture.

\begin{conj}[Cowling]\label{conj_Cow}
Every weakly amenable group with Cowling--Haagerup constant 1 has the Haagerup property.
\end{conj}

This conjecture was originally formulated as an equivalence; see \cite[\S 1.3.1]{CCJJV}. Later, the results in \cite{CoStVa, Oza, OzaPop} provided many examples of groups with the Haagerup property that are not weakly amenable. One such group is the wreath product $(\mathbb{Z}/2\mathbb{Z})\wr\mathbb{F}_2$. Conjecture \ref{conj_Cow} has been verified for every known example of group with $\boldsymbol\Lambda(\Gamma)=1$, which includes groups acting properly on finite-dimensional $\operatorname{CAT}(0)$ cube complexes \cite{GueHig, Miz, NibRee} and Baumslag--Solitar groups \cite{GalJan}. However, it remains open in general. We prove the following variation of Conjecture \ref{conj_Cow}. Its proof is based on Knudby's work on semigroups of Herz--Schur multipliers \cite{Knu}.

\begin{thm}\label{Thm_WA}
Every weakly amenable group with Cowling--Haagerup constant 1 admits a proper uniformly Lipschitz affine action on a subspace of an $L^1$ space.
\end{thm}

Our last application of Theorem \ref{Thm_caract} concerns a-TTT-menable groups, as defined in \cite{Oza2}. Let $\mathcal{H}$ be a (complex) Hilbert space and let $\mathcal{U}(\mathcal{H})$ denote its unitary group. A map $b:\Gamma\to\mathcal{H}$ is called a \textit{wq-cocycle} if there is a map (not necessarily a representation) $\pi:\Gamma\to\mathcal{U}(\mathcal{H})$ such that
\begin{align*}
\sup_{s,t\in\Gamma}\lVert b(st)-\pi(s)b(t)-b(s)\rVert < \infty.
\end{align*}
A group is said to be \textit{a-TTT-menable} if it admits a proper wq-cocycle. Since every cocycle is a wq-cocycle, every a-T-menable group is a-TTT-menable. Moreover, all hyperbolic groups are a-TTT-menable; see \cite[\S 1]{Oza2}. We show that wq-cocycles naturally give rise to almost invariant CND kernels. As a consequence, we obtain the following.

\begin{thm}\label{Thm_a-TTT}
Every a-TTT-menable group admits a proper uniformly Lipschitz affine action on a subspace of an $L^1$ space.
\end{thm}

Since Theorem \ref{Thm_caract} provides a characterisation, it is natural to ask which classes of groups do not admit proper almost invariant CND kernels. From the definition, we see that the map $f$ in \eqref{def_CND} is a coarse embedding into $\mathcal{H}$; we refer the reader to \cite[\S 1]{Wil} for details on coarse embeddings. Therefore the existence of a proper almost invariant CND kernel implies the existence of a coarse embedding into a Hilbert space. Although this is a very weak property, it is known that there are groups that are not coarsely embeddable. The study of such groups was initiated in \cite{Gro}. Surprisingly, this seems to be the only known obstruction. In particular, we do not know whether $\operatorname{SL}(3,\mathbb{Z})$ admits a proper almost invariant CND kernel; or equivalently, whether it has a proper uniformly Lipschitz affine action on a subspace of an $L^1$ space. This marks a sharp contrast with the case of $E\subseteq L^p$ with $1<p<\infty$, for which we know the answer is negative; see \cite[Corollary 1.3]{dlS} or \cite[Theorem A]{dLadlS}.

This article is organised as follows. In Section \ref{S_ker}, we give the definitions and basic facts about the different kinds of kernels that we consider in the paper. Theorem \ref{Thm_caract} and Corollary \ref{Cor_Lp} are also proved in that section. Section \ref{S_QT} is devoted to groups acting on products of quasi-trees and two different proofs of Theorem \ref{Thm_QT}. Theorems \ref{Thm_WA} and \ref{Thm_a-TTT} are proved in Sections \ref{S_WA} and \ref{S_a-TTT} respectively.

\subsection*{Acknowledgements}
I am grateful to Mikael de la Salle for many interesting discussions and for his very valuable feedback since an early stage of this work. I also thank the anonymous referees of a previous version of this paper, whose comments and suggestions led to several improvements.

\section{{\bf CND kernels and actions on $E\subseteq L^p$}}\label{S_ker}

In this section we prove Theorem \ref{Thm_caract} and Corollary \ref{Cor_Lp}. We begin by giving the definitions and main characterisations of positive definite kernels and conditionally negative definite kernels. For a more detailed treatment, we refer the reader to \cite[Appendix C]{BedlHVa}.

\subsection{Kernels}
Let $X$ be a set. A function $\kappa:X\times X\to\mathbb{C}$ is said to be a \textit{positive definite kernel} on $X$ if, for every finite sequence $z_1,\ldots,z_n\in\mathbb{C}$ and $x_1,\ldots,x_n\in X$,
\begin{align*}
\sum_{i,j=1}^n\kappa(x_i,x_j)z_i\overline{z_j}\geq 0.
\end{align*}
If $\xi:X\to\mathcal{H}$ is a map from $X$ to a Hilbert space, then the kernel
\begin{align*}
\kappa(x,y)=\langle\xi(x),\xi(y)\rangle
\end{align*}
is positive definite. Moreover, every positive definite kernel on $X$ is of this form.

A symmetric function $\psi:X\times X\to\mathbb{R}$ is said to be a \textit{conditionally negative definite kernel} on $X$ if it vanishes on the diagonal ($\psi(x,x)=0,\ \forall x\in X$), and for every finite sequence $\alpha_1,\ldots,\alpha_n\in\mathbb{R}$ satisfying
\begin{align*}
\sum_{i=1}^n\alpha_i=0,
\end{align*}
and every $x_1,\ldots,x_n\in X$, we have
\begin{align*}
\sum_{i,j=1}^n\psi(x_i,x_j)\alpha_i\alpha_j\leq 0.
\end{align*}
If $b:X\to\mathcal{H}$ is a map from $X$ to a Hilbert space, then the kernel
\begin{align*}
\psi(x,y)=\lVert b(x)-b(y)\rVert ^2
\end{align*}
is CND. Moreover, every CND kernel on $X$ is of this form. A very important result relating these two objects is Schoenberg's theorem \cite{Sch}. 

\begin{thm}[Schoenberg]\label{Thm_Scho}
Let $X$ be a set and let $\psi:X\times X\to\mathbb{R}$ be a symmetric function such that $\psi(x,x)=0$ for all $x\in X$. Then the following are equivalent:
\begin{itemize}
\item[(i)] The kernel $\psi$ is conditionally negative definite.
\item[(ii)] For all $\lambda>0$, the kernel $e^{-\lambda\psi}$ is positive definite.
\end{itemize}
\end{thm}

The following result is a particular case of \cite[Lemma 2.4]{CoTeVa}, which in turn is a reformulation of \cite[Theorem 8]{Sch2} in a more modern language. We include its proof for completeness.

\begin{lem}\label{Lem_log_CND}
Let $X$ be a set and let $\psi_0:X\times X\to[0,\infty)$ be a CND kernel. Then the kernel $\psi:X\times X\to[0,\infty)$ given by
\begin{align*}
\psi(x,y)=\log(1+\psi_0(x,y)),\quad\forall x,y\in X
\end{align*}
is also CND.
\end{lem}
\begin{proof}
For all $r\geq 0$, we can write
\begin{align*}
\log(1+r)=\int_0^r\int_0^\infty e^{-t(1+u)}\,dt\,du=\int_0^\infty\frac{1-e^{-tr}}{t}e^{-t}\, dt.
\end{align*}
Since $\psi_0$ is CND, by Theorem \ref{Thm_Scho}, so is $1-e^{-t\psi_0}$ for all $t>0$. Since the set of CND kernels on $X$ is a convex cone, the identity above tells us that $\log(1+\psi_0)$ is CND.
\end{proof}

\subsection{$L^p$ spaces}
Let $(\Omega,\mu)$ be a measure space and let $p\in(0,\infty)$. The space $L^p(\Omega,\mu)$ is the vector space of (equivalence classes of) measurable functions $f:\Omega\to\mathbb{R}$ such that
\begin{align*}
\lVert f\rVert_p=\bigg(\int_\Omega |f(\omega)|^p\,d\mu(\omega)\bigg)^{1/p} <\infty.
\end{align*}
For $p\in[1,\infty)$, $L^p(\Omega,\mu)$ is a Banach space. For $p\in(0,1)$, this is no longer the case, but we still get a metric space for the distance
\begin{align*}
(f,g)\mapsto\lVert f-g\rVert_p^p,
\end{align*}
so it makes sense to speak about proper actions. We say that $X$ is an $L^p$ space if $X=L^p(\Omega,\mu)$ for some measure space $(\Omega,\mu)$. We record here the following classical result; see \cite[Lemma 7]{Now3} for a proof.

\begin{prop}\label{Prop_Lp_CND}
Let $p\in(0,2]$ and let $X$ be an $L^p$ space. Then the map
\begin{align*}
(f,g)\mapsto\lVert f-g\rVert_p^p
\end{align*}
is a CND kernel on $X$.
\end{prop}

\subsection{Proof of Theorem \ref{Thm_caract}}
We will start by proving the implication (ii)$\implies$(i) in Theorem \ref{Thm_caract}, which essentially follows from Lemma \ref{Lem_log_CND} and Proposition \ref{Prop_Lp_CND}.

\begin{lem}\label{Lem(ii)(i)}
Let $\Gamma$ be a countable group admitting a proper uniformly Lipschitz affine action on a subspace of an $L^1$ space. Then there exists a proper almost invariant CND kernel $\psi:\Gamma\times\Gamma\to[0,\infty)$.
\end{lem}
\begin{proof}
Let $\Gamma\curvearrowright^\sigma E$ be the proper uniformly Lipschitz action given by the hypothesis, and take $v_0\in E$. By Proposition \ref{Prop_Lp_CND}, the kernel
\begin{align*}
\psi_0(x,y)=\lVert \sigma(x)v_0-\sigma(y)v_0\rVert,\quad\forall x,y\in\Gamma,
\end{align*}
is CND. Let $C_0\geq 1$ be the Lipschitz constant of $\sigma$. Then, for all $s,x,y\in\Gamma$,
\begin{align*}
\psi_0(sx,sy) \leq C_0\psi_0(x,y).
\end{align*}
Moreover, by Lemma \ref{Lem_log_CND}, the kernel
\begin{align*}
\psi=\log(1+\psi_0)
\end{align*}
is CND too. This kernel satisfies
\begin{align*}
\psi(sx,sy)\leq\psi(x,y)+C,\quad\forall s,x,y\in\Gamma,
\end{align*}
with $C=\log C_0\geq 0$. Moreover, the map $x\mapsto\psi(x,e)$ is proper because $\sigma$ is proper.
\end{proof}

The implication (i)$\implies$(ii)' is based on a variation of the GNS construction. Our main reference for this is \cite[Theorem C.2.3]{BedlHVa}.

\begin{lem}\label{Lem(i)(ii)'}
Let $\Gamma$ be a countable group such that there exists a proper almost invariant CND kernel $\psi:\Gamma\times\Gamma\to[0,\infty)$. Then, for every $\varepsilon>0$, there exist a measure space $(\Omega,\mu)$, a closed subspace $E\subseteq L^1(\Omega,\mu)$, and a proper $(1+\varepsilon)$-Lipschitz affine action of $\Gamma$ on $E$.
\end{lem}
\begin{proof}
Let $\varepsilon>0$. Recall that, by definition, there is $C\geq 0$ such that
\begin{align*}
\psi(sx,sy)\leq\psi(x,y) + C,\quad\forall s,x,y\in\Gamma.
\end{align*}
Since $(\varepsilon^2/C)\psi$ is also a proper almost invariant CND kernel, by rescaling, we may assume that
\begin{align}\label{psi_eps}
\psi(sx,sy)\leq\psi(x,y) + \varepsilon^2,\quad\forall s,x,y\in\Gamma.
\end{align}
Let $V$ be the vector space of finitely supported, real-valued functions on $\Gamma$ with mean $0$. In other words, for every $v\in V$,
\begin{align*}
\sum_{x\in \Gamma}v(x)=0,
\end{align*}
where only finitely many summands are non-zero. We define a scalar product on $V$ as follows. For $v,w\in V$, we set
\begin{align}\label{def_prod_a}
\langle v,w\rangle_\psi = -\sum_{x,y\in \Gamma}v(x)w(y)\psi(x,y) + \sum_{x\in \Gamma}v(x)w(x).
\end{align}
Notice that this is indeed a positive definite bilinear form because $\psi$ is CND. The last term in \eqref{def_prod_a} ensures that $\langle v,v\rangle_\psi>0$ for all $v\in V\setminus\{0\}$. We define a Hilbert space $\mathcal{H}_\psi$ as the completion of $V$ for $\langle\cdot,\cdot\rangle_\psi$. Now let $E$ be the subspace of $\mathcal{H}_\psi$ of all those $v$ such that
\begin{align*}
\lVert v\rVert_1=\sum_{x\in \Gamma}|v(x)| < \infty.
\end{align*}
We endow $E$ with the norm
\begin{align}\label{norm_E}
\lVert v\rVert_E=\lVert v\rVert_\psi + \lVert v\rVert_1,
\end{align}
for which it becomes a Banach space. Since $\mathcal{H}_\psi$ is an $L^2$ space, it embeds linearly and isometrically into an $L^1$ space; see \cite[Th\'eor\`eme 2]{BrDaKr} or \cite[Theorem 6.8]{ChDrHa}. More precisely, there is a measure space $(\Omega',\mu')$ and a linear isometry $J:\mathcal{H}_\psi\to L^1(\Omega',\mu')$. Now let us define a new map $\tilde{J}:E\to L^1(\Omega',\mu')\oplus_1\ell^1(\Gamma)$ by
\begin{align*}
v\mapsto (Jv,v),\quad\forall v\in E,
\end{align*}
where $\oplus_1$ stands for the $\ell^1$-direct sum. We can identify $L^1(\Omega',\mu')\oplus_1\ell^1(\Gamma)\cong L^1(\Omega,\mu)$, where $\Omega=\Omega'\sqcup\Gamma$ and $\mu$ is the sum of $\mu'$ and counting measure on $\Gamma$. Then $\tilde{J}$ is a linear isometry into $L^1(\Omega,\mu)$. We define a representation $\pi$ of $\Gamma$ on $V$ by
\begin{align}\label{def_rep}
\pi(s)v(x)=v(s^{-1}x),\quad\forall s\in\Gamma,\ \forall v\in V.
\end{align}
Observe that 
\begin{align*}
\lVert \pi(s)v\rVert_\psi^2 = - \sum_{x,y\in \Gamma}v(x)v(y)\psi(sx,sy) + \sum_{x\in \Gamma} v(x)^2.
\end{align*}
Thus
\begin{align*}
\lVert \pi(s)v\rVert_\psi^2 -\lVert v\rVert_\psi^2 = \sum_{x,y\in \Gamma}v(x)v(y)(\psi(x,y)-\psi(sx,sy)).
\end{align*}
Moreover, we know from \eqref{psi_eps} that
\begin{align*}
|\psi(x,y)-\psi(sx,sy)|\leq \varepsilon^2,\quad\forall s,x,y\in\Gamma,
\end{align*}
which implies that
\begin{align*}
\lVert \pi(s)v\rVert_\psi^2 -\lVert v\rVert_\psi^2 \leq \varepsilon^2 \sum_{x,y\in \Gamma}|v(x)| |v(y)| = \varepsilon^2 \lVert v\rVert_1^2. 
\end{align*}
Therefore
\begin{align*}
\lVert \pi(s)v\rVert_\psi &\leq \Big(\varepsilon^2 \lVert v\rVert_1^2 + \lVert v\rVert_\psi^2\Big)^{1/2}\\
&\leq \varepsilon\lVert v\rVert_1+\lVert v\rVert_\psi.
\end{align*}
We conclude that
\begin{align*}
\lVert \pi(s)v\rVert_E &= \lVert \pi(s)v\rVert_\psi + \lVert \pi(s)v\rVert_1 \\
&\leq (\varepsilon+1)\big(\lVert v\rVert_\psi+\lVert v\rVert_1\big)\\
&=(\varepsilon+1)\lVert v\rVert_E.
\end{align*}
This shows that $\pi$ extends to a uniformly bounded representation on $E$ with uniform bound $(1+\varepsilon)$. We only need to construct a proper cocycle to conclude. Let us define $b:\Gamma\to E$ by
\begin{align*}
b(s)=\delta_s-\delta_e,\quad\forall s\in\Gamma,
\end{align*}
where
\begin{align*}
\delta_x(y)=\begin{cases}
1 & y=x\\ 0 & y\neq x.
\end{cases}
\end{align*}
Observe that $b(s)$ belongs to $V$ for all $s\in\Gamma$, so it is well defined as an element of $E$. Moreover,
\begin{align*}
b(st)&=\pi(s)(\delta_t-\delta_e)+\delta_s-\delta_e\\
&=\pi(s)b(t)+b(s),
\end{align*}
for all $s,t\in\Gamma$, so $b$ is indeed a cocycle. Finally,
\begin{align*}
\lVert b(s)\rVert_E &= \lVert \delta_s-\delta_e\rVert_\psi + \lVert \delta_s-\delta_e\rVert_1\\
&\geq (2\psi(s,e)+2)^{1/2}.
\end{align*}
This shows that $b$ is proper because $\psi$ is proper.
\end{proof}

Lemmas \ref{Lem(ii)(i)} and \ref{Lem(i)(ii)'} prove the implications (ii)$\implies$(i)$\implies$(ii)' in Theorem \ref{Thm_caract}. Since (ii)' clearly implies (ii), we conclude that the three conditions are equivalent. This finishes the proof of the theorem. 

The proof of Corollary \ref{Cor_Lp} uses the same ideas as that of Lemma \ref{Lem(ii)(i)}.

\begin{proof}[Proof of Corollary \ref{Cor_Lp}]
Let $p\in(0,2]$ and let $A$ be a subset of an $L^p$ space such that there is a proper uniformly Lipschitz action $\Gamma\curvearrowright^\sigma A$. Take $v_0\in A$ and define
\begin{align*}
\psi_0(x,y)=\lVert \sigma(x)v_0-\sigma(y)v_0\rVert^p,\quad\forall x,y\in\Gamma.
\end{align*}
By Proposition \ref{Prop_Lp_CND}, $\psi_0$ is a CND kernel. Let $C_0\geq 1$ be the Lipschitz constant of $\sigma$. Then, for all $s,x,y\in\Gamma$,
\begin{align*}
\psi_0(sx,sy) \leq C_0^p\psi_0(x,y).
\end{align*}
This shows that the CND kernel $\psi=\log(1+\psi_0)$ satisfies
\begin{align*}
\psi(sx,sy)\leq\psi(x,y)+C,\quad\forall s,x,y\in\Gamma,
\end{align*}
with $C=p\log C_0$. The map $x\mapsto\psi(x,e)$ is proper because $\sigma$ is proper. By Theorem \ref{Thm_caract}, $\Gamma$ has a proper uniformly Lipschitz affine action on $E\subseteq L^1$.
\end{proof}

\begin{rmk}
In the case that $A$ is a subspace (not merely a subset) of an $L^p$ space with $p\in[1,2]$, this result was known, and it is simply a consequence of the fact that every $L^p$ space embeds isometrically into an $L^1$ space; see \cite[Th\'eor\`eme 2]{BrDaKr} or \cite[Theorem 6.8]{ChDrHa}. For the other cases, we believe it is new.
\end{rmk}

\section{{\bf Groups acting on products of quasi-trees}}\label{S_QT}

This section is devoted to the proof of Theorem \ref{Thm_QT}. We shall give two proofs of this result. The first one is based on Kerr's characterisation of quasi-trees in terms of $(1,C)$-quasi-isometries \cite{Ker}. The second one is inspired by Bo\.zejko and Picardello's work on Schur multipliers on trees \cite{BozPic}, and its subsequent generalisations.

\subsection{Quasi-isometries and the first proof}

We begin by defining quasi-isometries and quasi-trees. For a more detailed treatment, we refer the reader to \cite{Man} and \cite{Ker}.

Let $(X,d_X)$ and $(Y,d_Y)$ be two metric spaces and let $L\geq 1$, $C>0$. We say that $f:X\to Y$ is an $(L,C)$-\textit{quasi-isometry} if
\begin{itemize}
\item[1)] For all $x_1,x_2\in X$,
\begin{align*}
\frac{1}{L}d_X(x_1,x_2)-C\leq d_Y(f(x_1),f(x_2))\leq L d_X(x_1,x_2)+C.
\end{align*}
\item[2)] For all $y\in Y$, there is $x\in X$ such that $d_Y(f(x),y)\leq C$.
\end{itemize}
If such a map exists, we say that $(X,d_X)$ and $(Y,d_Y)$ are $(L,C)$-\textit{quasi-isometric}. We say that they are \textit{quasi-isometric} if they are $(L,C)$-quasi-isometric for some $L\geq 1$, $C>0$.

We will view a connected graph as a metric space $(X,d)$. This means that $X$ is the set of vertices of the graph, and $d$ is the edge-path distance on it. A connected graph is called a \textit{tree} if every two vertices are connected by exactly one simple path. It is a \textit{quasi-tree} if it is quasi-isometric to a tree. The following was proved in \cite[Theorem 1.3]{Ker}.

\begin{thm}[Kerr]\label{Thm_Kerr}
If $X$ is a quasi-tree, then there exists $C\geq 0$ such that $X$ is $(1,C)$-quasi-isometric to a tree.
\end{thm}

The following result was first proved by Haagerup for free groups; see \cite[Lemma 1.2]{Haa}. It is the reason why we speak of the \textit{Haagerup property}. It was extended to trees in \cite{Wat}; see also \cite[Lemma 2.3]{JulVal}.

\begin{thm}[Haagerup, Watatani]\label{Thm_Haa_Wat}
The edge-path distance on a tree is a CND kernel.
\end{thm}

Now we can give our first proof of Theorem \ref{Thm_QT}.

\begin{proof}[First proof of Theorem \ref{Thm_QT}]
Let $(X_1,d_1),\ldots,(X_N,d_N)$ be $N$ quasi-trees. By Theorem \ref{Thm_Kerr}, for each $i\in\{1,\ldots,N\}$, there is a $(1,C_i)$-quasi-isometry $g_i:X_i\to T_i$, where $T_i$ is a tree and $C_i\geq 0$. Let $X=X_1\times\cdots\times X_N$ and let $d_X$ be the $\ell^1$-combination of the metrics $d_1,\ldots d_N$ as in \eqref{l1_combi}. We define $(T,d_T)$ similarly. Let $g:X\to T$ be given by
\begin{align*}
g(x_1,\ldots,x_N)=(g_1(x_1),\ldots,g_N(x_N)),\quad\forall (x_1,\ldots,x_N)\in X.
\end{align*}
We define $\psi:X\times X\to[0,\infty)$ by
\begin{align*}
\psi(x,y)=d_T(g(x),g(y)),\quad\forall x,y\in X.
\end{align*}
By Theorem \ref{Thm_Haa_Wat}, $\psi$ is the sum of CND kernels, and therefore it is CND too. Now let $\Gamma$ be a group acting properly by isometries on $(X,d_X)$, and fix $x_0\in X$. By restricting $\psi$ to the orbit $\Gamma\cdot x_0$, we can view it as a CND kernel on $\Gamma$. We claim that it is almost invariant and proper. Indeed, defining $C=C_1+\cdots+C_N$, we have
\begin{align}\label{ineq_g}
d_X(x,y)-C\leq d_T(g(x),g(y)) \leq d_X(x,y)+C,\quad\forall x,y\in X.
\end{align}
In particular,
\begin{align*}
\psi(s\cdot x,s\cdot y)\leq\psi(x,y)+2C,
\end{align*}
for all $x,y\in\Gamma\cdot x_0$ and $s\in\Gamma$. Finally, the inequality
\begin{align*}
\psi(s\cdot x_0,x_0)\geq d_X(s\cdot x_0,x_0) - C
\end{align*}
shows that $\psi$ is proper. By Theorem \ref{Thm_caract}, we conclude that $\Gamma$ has a proper uniformly Lipschitz affine action on a subspace of an $L^1$ space.
\end{proof}

\begin{rmk}
The previous proof does not make use of the graph structure on $X_i$, only that it is $(1,C)$-quasi-isometric to a tree. Although we only deal with graphs here, quasi-trees can be defined more generally as geodesic metric spaces that are quasi-isometric to trees. Since the results in \cite{Ker} still hold in that level of generality, the previous proof works in that context too.
\end{rmk}

\subsection{The kernel $d_a$ and the second proof}
For the second proof of Theorem \ref{Thm_QT}, we will use a construction that makes sense for any connected graph, but which becomes particularly interesting for quasi-trees. As before, the main goal is to construct an almost invariant CND kernel on a product of quasi-trees.

Let $(X,d)$ be a connected graph. For $a\in X$ and $k\in\mathbb{N}$, let
\begin{align*}
B(a,k)=\{x\in X\,\mid\, d(a,x)\leq k\}.
\end{align*}
The following definition was given in (an earlier version of) \cite{CDHL}. For every $x,y\in X$, we set
\begin{align}\label{def_R_a}
R_a(x,y)=\min\Bigg\{k\geq 0\,\Bigm\vert \begin{array}{c}
\text{There is no connected component of }\\ X\setminus B(a,k) \text{ containing both }x\text{ and } y.
\end{array} \Bigg\}.
\end{align}
Since $x$ does not belong to $X\setminus B(a,d(a,x))$, we always have
\begin{align*}
R_a(x,y)\leq \min\{d(a,x),d(a,y)\}.
\end{align*}

The following result is known as the bottleneck property. It can be thought of as a stronger form of hyperbolicity for quasi-trees; see \cite[Lemma 2.16]{Man}.

\begin{lem}[Manning]\label{Lem_Manning}
Let $(X,d)$ be a quasi-tree. There exists $\Delta>0$ such that $X$ satisfies the following: Let $x,y\in X$, and let $p$ be a point lying in a geodesic between $x$ and $y$. Then every path joining $x$ to $y$ must pass within $\Delta$ of $p$. In particular, $X$ is $\Delta$-hyperbolic.
\end{lem}

Here, by $\delta$-hyperbolic space, we mean a geodesic metric space such that all the geodesic triangles are $\delta$-thin; see \cite[Definition 2.6]{Man}.

For $a,x,y\in X$, the Gromov product $(x,y)_a$ is defined as
\begin{align*}
(x,y)_a=\tfrac{1}{2}(d(a,x)+d(a,y)-d(x,y)).
\end{align*}
We will need the following lemma. For a proof, see for instance \cite[Lemme 17]{GhydlH}.

\begin{lem}\label{Lem_Gom_prod}
Let $(X,d)$ be a $\delta$-hyperbolic space. Let $a,x,y\in X$ and let $[x,y]$ be a geodesic path between $x$ and $y$. Then
\begin{align*}
(x,y)_a \leq d(a,[x,y]) \leq (x,y)_a + \delta.
\end{align*}
\end{lem}

The following lemma shows that, on a quasi-tree, $R_a(x,y)$ coincides with the Gromov product $(x,y)_a$ up to bounded error.

\begin{lem}\label{Lem_Delta'}
Let $(X,d)$ be a quasi-tree and let $\Delta>0$ be the constant given by Lemma \ref{Lem_Manning}. Then, for all $a,x,y\in X$,
\begin{align*}
d(x,y)-4\Delta\leq d(a,x)+d(a,y)-2R_a(x,y)\leq d(x,y).
\end{align*}
\end{lem}
\begin{proof}
We begin with the second inequality, which is independent of the geometry of the graph. Let $\gamma$ be a geodesic path joining $x$ to $y$, and let $u$ be the vertex in $\gamma$ closest to $a$. Observe that $d(a,u)\leq R_a(x,y)$. Therefore
\begin{align*}
d(a,x)+d(a,y) &\leq d(a,u) + d(u,x) + d(a,u) + d(u,y)\\
&\leq 2R_a(x,y) + d(x,y).
\end{align*}
Now we turn to the first inequality. Let $u$ be as above and let $\gamma'$ be any path joining $x$ to $y$. By Lemma \ref{Lem_Manning}, there is a vertex $w$ in $\gamma'$ such that $d(u,w)\leq\Delta$. This shows that $\gamma'$ is not contained in $X\setminus B(a,d(a,u)+\Delta)$. Hence
\begin{align*}
R_a(x,y)\leq d(a,u) + \Delta.
\end{align*}
On the other hand, by Lemma \ref{Lem_Gom_prod},
\begin{align*}
d(a,u) \leq (x,y)_a + \Delta.
\end{align*}
Hence
\begin{align*}
2 R_a(x,y)&\leq 2d(a,u)+2\Delta\\
&\leq d(a,x) + d(a,y) - d(x,y) + 4 \Delta.
\end{align*}
\end{proof}

Given a connected graph $(X,d)$ and $a\in X$, we define $d_a:X\times X\to\mathbb{N}$ by
\begin{align}\label{def_d_a}
d_a(x,y)= d(a,x)+d(a,y)-2R_a(x,y),\quad\forall x,y\in X.
\end{align}
Observe that $d_a$ is a symmetric kernel because $R_a$ is symmetric. Moreover, it vanishes on the diagonal because $R_a(x,x)=d(a,x)$ for all $x$. For quasi-trees, the inequality given by Lemma \ref{Lem_Delta'} can be written as
\begin{align}\label{ineq_d_a}
d(x,y)-4\Delta\leq d_a(x,y)\leq d(x,y),\quad\forall a,x,y\in X.
\end{align}
This inequality should be compared with \eqref{ineq_g} in the first proof of Theorem \ref{Thm_QT}. The kernel $d_a$ will play the role of the pull-back of the metric on a tree. Therefore, our next task is proving that $d_a$ is a CND kernel. We will see that this is true for any connected graph, and we will achieve it by showing that the kernel $(x,y)\mapsto r^{d_a(x,y)}$ is positive definite for all $r\in(0,1)$.

\begin{lem}\label{Lem_S_a(x)}
Let $(X,d)$ be a connected graph. For all $a,x\in X$, we have the following equality of sets:
\begin{align*}
\{R_a(x,y)\,\mid\, y\in X\}=\{0,...,d(a,x)\},
\end{align*}
where $R_a(x,y)$ is defined as in \eqref{def_R_a}.
\end{lem}
\begin{proof}
Let $\gamma:\{0,...,d(a,x)\}\to X$ be a geodesic path joining $a$ to $x$. Then
\begin{align*}
R_a(\gamma(k),x)=k,
\end{align*}
for all $k\in\{0,...,d(a,x)\}$. Since $R_a(x,y)\leq d(a,x)$ for every $y\in X$, these are the only possible values.
\end{proof}

For all $a,x\in X$, $k\in\mathbb{N}$, let us define the set
\begin{align}\label{def_V_a(x,k)}
\mathcal{V}_a(x,k)=\{y\in X\,\mid\, R_a(x,y)\geq k\}.
\end{align}
This may be viewed as the ball of center $x$ and radius $e^{-k}$ for the metric $\sphericalangle_a$ given by
\begin{align*}
\sphericalangle_a(x,y)=\begin{cases}
e^{-R_a(x,y)}, & y\neq x\\
0, & y=x.
\end{cases}
\end{align*}
It was proven in (the first version of) \cite{CDHL} that $\sphericalangle_a$ satisfies the ultrametric inequality. We reproduce the proof here for completeness.

\begin{lem}
Let $(X,d)$ be a connected graph. For all $a,x,y,z\in X$,
\begin{align*}
R_a(x,y)\geq\min\{R_a(x,z),R_a(z,y)\}.
\end{align*}
In particular,
\begin{align}\label{um_ineq}
\sphericalangle_a(x,y)\leq\max\{\sphericalangle_a(x,z),\sphericalangle_a(z,y)\}.
\end{align}
\end{lem}
\begin{proof}
Let $a,x,y,z\in X$. We will assume that $\min\{R_a(x,z),R_a(z,y)\}>0$ since otherwise there is nothing to prove. Take $k=\min\{R_a(x,z),R_a(z,y)\}-1$. By definition, $x$ and $z$ belong to the same connected component of $X\setminus B(a,k)$. The same holds for $z$ and $y$, and therefore $x$ and $y$ belong to the same connected component of $X\setminus B(a,k)$. This implies that $R_a(x,y)>k$. Equivalently
\begin{align*}
R_a(x,y)\geq\min\{R_a(x,z),R_a(z,y)\}.
\end{align*}
The second inequality follows from the fact that the function $t\mapsto e^{-t}$ is decreasing.
\end{proof}

\begin{lem}\label{Lem_B_a=B_a}
Let $(X,d)$ be a connected graph. For all $x,y\in X$ and all $k\in\{0,...,d(a,x)\}$, $j\in\{0,...,d(a,y)\}$,
\begin{align*}
\mathcal{V}_a(x,k)=\mathcal{V}_a(y,j) \iff k=j \text{ and } R_a(x,y)\geq k,
\end{align*}
where $\mathcal{V}_a$ is defined as in \eqref{def_V_a(x,k)}.
\end{lem}
\begin{proof}
Recall that the ultrametric inequality \eqref{um_ineq} implies that every element of a ball is its centre. If we assume that $k=j$ and $R_a(x,y)\geq k$, then $x$ belongs to $\mathcal{V}_a(y,k)$, which implies that $\mathcal{V}_a(x,k)=\mathcal{V}_a(y,k)$. Conversely, assume that $\mathcal{V}_a(x,k)=\mathcal{V}_a(y,j)$. Without loss of generality, we may suppose that $k\leq j$. Since $y$ belongs to $\mathcal{V}_a(x,k)$, the ultrametric inequality tells us that
\begin{align*}
\mathcal{V}_a(y,k)=\mathcal{V}_a(x,k)=\mathcal{V}_a(y,j).
\end{align*}
By Lemma \ref{Lem_S_a(x)}, $k=j$. This finishes the proof.
\end{proof}

Now we are ready to prove that $(x,y)\mapsto r^{d_a(x,y)}$ is positive definite. The proof is inspired by \cite[Proposition 2.1]{BozPic}, where the authors estimate the norms of certain Schur multipliers on trees. The idea is to fix a point at infinity, and then, for each vertex, follow a geodesic ray towards that fixed point. This idea has been extremely fruitful in the study of Schur multipliers and weak amenability, and it has been adapted to several more general contexts. In our case, we seek to construct a positive definite kernel, which is a particular case of Schur multiplier; see \cite[\S 2]{BozPic}. Intuitively, what we will do is fix a vertex $a$ and attach to it an infinite geodesic ray. This will play the role of the point at infinity.

\begin{prop}\label{Prop_pos_def}
Let $(X,d)$ be a connected graph and $a\in X$. Let $d_a:X\times X\to\mathbb{N}$ be the kernel defined in \eqref{def_d_a}. For all $r\in(0,1)$, the kernel
\begin{align*}
(x,y)\in X^2 \mapsto r^{d_a(x,y)}
\end{align*}
is positive definite.
\end{prop}
\begin{proof}
Let $a\in X$ and $r\in(0,1)$. For each $x\in X$ and $k\in\{0,\ldots,d(a,x)\}$, define
\begin{align*}
\Omega(x,k)=\mathcal{V}_a(x,d(x,a)-k),
\end{align*}
where $\mathcal{V}_a$ is defined as in \eqref{def_V_a(x,k)}. In the case of a tree, if we denote by $x=x_0,\ldots,x_n=a$ the unique geodesic path joining $x$ to $a$, $\Omega(x,k)$ is the half-space containing $x_k$, but not $x_{k+1}$. Hence we recover the idea of following a geodesic path towards a fixed point. Let $\mathcal{P}(X)$ denote the power set of $X$. Recall that, for any set $Y$, the Hilbert space $\ell^2(Y)$ has a canonical orthonormal basis $(\delta_y)_{y\in Y}$ defined by
\begin{align*}
\delta_y(w)=\begin{cases}
1 & w=y\\ 0 & w\neq y.
\end{cases}
\end{align*}
Define $\xi:X\to\ell^2(\mathcal{P}(X))\oplus\ell^2(\mathbb{N})$ by
\begin{align*}
\xi(x)&=(1-r^2)^{\frac{1}{2}}\Bigg(\sum_{k=0}^{d(a,x)-1}r^{k}\delta_{\Omega(x,k)}, \sum_{k\geq d(a,x)}r^{k}\delta_{k-d(a,x)}\Bigg),
\end{align*}
for all $x\in X$. Here $\ell^2(\mathbb{N})$ can be viewed as the $\ell^2$ space of an infinite geodesic ray attached to $a$. We have, for all $x,y\in X$,
\begin{align*}
\langle \xi(x),\xi(y)\rangle &=(1-r^2)\sum_{k=0}^{d(a,x)-1}\sum_{j=0}^{d(a,y)-1}r^{k+j}\langle\delta_{\Omega(x,k)},\delta_{\Omega(y,j)}\rangle\\
&\qquad + (1-r^2)\sum_{k\geq d(a,x)}\sum_{j\geq d(a,y)}r^{k+j}\langle\delta_{k-d(a,x)},\delta_{j-d(a,y)}\rangle.
\end{align*}
By Lemma \ref{Lem_B_a=B_a}, $\Omega(x,k)=\Omega(y,j)$ if and only if
\begin{align*}
d(a,x)-k=d(a,y)-j \qquad\text{and}\qquad R_a(x,y)\geq d(a,x)-k.
\end{align*}
Therefore
\begin{align*}
\sum_{k=0}^{d(a,x)-1}\sum_{j=0}^{d(a,y)-1}r^{k+j}\langle\delta_{\Omega(x,k)},\delta_{\Omega(y,j)}\rangle
&= \sum_{k=d(a,x)-R_a(x,y)}^{d(a,x)-1}r^{2k-d(a,x)+d(a,y)}\\
&= r^{d_a(x,y)}\sum_{k=0}^{R_a(x,y)-1}r^{2k}.
\end{align*}
On the other hand,
\begin{align*}
\sum_{k\geq d(a,x)}\sum_{j\geq d(a,y)}r^{k+j}\langle\delta_{k-d(a,x)},\delta_{j-d(a,y)}\rangle
&= \sum_{k\geq d(a,x)}r^{2k-d(a,x)+d(a,y)}\\
&= r^{d_a(x,y)}\sum_{k\geq R_a(x,y)}r^{2k}.
\end{align*}
This shows that
\begin{align*}
\langle \xi(x),\xi(y)\rangle &= (1-r^2) r^{d_a(x,y)}\sum_{k\geq 0}r^{2k}\\
&= r^{d_a(x,y)}.
\end{align*}
We conclude that $(x,y)\mapsto r^{d_a(x,y)}$ is a positive definite kernel.
\end{proof}

\begin{cor}\label{Cor_d_a_cnd}
Let $X$ be a connected graph. For all $a\in X$, the kernel
\begin{align*}
(x,y)\in X^2 \mapsto d_a(x,y)
\end{align*}
is conditionally negative definite.
\end{cor}
\begin{proof}
This follows directly from Proposition \ref{Prop_pos_def} and Theorem \ref{Thm_Scho}.
\end{proof}

Now we are ready to give our second proof of Theorem \ref{Thm_QT}.

\begin{proof}[Second proof of Theorem \ref{Thm_QT}]
Let $(X_1,d_1),\ldots,(X_N,d_N)$ be $N$ quasi-trees, and fix $a=(a_1,...,a_N)\in X$. For all $i\in\{1,\ldots,N\}$, let $d_{a_i}:X_i\to\mathbb{N}$ be the kernel defined as in \eqref{def_d_a}. By Corollary \ref{Cor_d_a_cnd}, $d_{a_i}$ is a CND kernel. Let $X=X_1\times\cdots\times X_N$, and define $\psi:X\times X\to\mathbb{N}$ by
\begin{align*}
\psi(x,y)=\sum_{i=1}^N d_{a_i}(x_i,y_i),\quad \forall x,y\in X.
\end{align*}
This kernel is CND because it is a sum of CND kernels. Moreover, by Lemma \ref{Lem_Delta'}, there is a constant $C\geq 0$ such that
\begin{align*}
d(x,y)-C \leq \psi(x,y)\leq d(x,y),\quad\forall x,y\in X,
\end{align*}
where
\begin{align*}
d(x,y)=\sum_{i=1}^Nd_i(x_i,y_i).
\end{align*}
In particular, if $\Gamma$ is a group acting properly by isometries on $X$, then
\begin{align*}
\psi(s\cdot x,s\cdot y)\leq\psi(x,y)+C,\quad\forall s\in\Gamma,\ \forall x,y\in X.
\end{align*}
Restricting $\psi$ to the orbit $\Gamma\cdot a$, we get a proper almost invariant CND kernel on $\Gamma$. By Theorem \ref{Thm_caract}, we obtain a proper uniformly Lipschitz affine action on $E\subseteq L^1$.
\end{proof}

\section{{\bf Weakly amenable groups}}\label{S_WA}

Now we focus on weakly amenable groups and the proof of Theorem \ref{Thm_WA}. For this purpose, we shall look at a weaker property introduced by Knudby \cite{Knu}. We say that a countable group $\Gamma$ has the weak Haagerup property if there exists a sequence of functions $\varphi_n:\Gamma\to\mathbb{C}$ vanishing at infinity and converging pointwise to 1 such that $\sup_n\lVert \varphi_n\rVert _{B_2(\Gamma)}< \infty$. We will only be interested in the weak Haagerup property with constant 1, meaning that the condition above holds with $\sup_n\lVert \varphi_n\rVert _{B_2(\Gamma)}\leq 1$. 
We will not recall the precise definition of the norm $\lVert \cdot\rVert _{B_2(\Gamma)}$ because we will not use it in our argument. Instead, we refer the interested reader to \cite[Appendix D]{BroOza}.

Since every finitely supported function vanishes at infinity, every weakly amenable group $\Gamma$ with $\boldsymbol\Lambda(\Gamma)=1$ has the weak Haagerup property with constant 1. The following is a consequence of \cite[Theorem 1.2]{Knu} and \cite[Proposition 4.3]{Knu}.

\begin{thm}[Knudby]\label{Thm_Knu}
Let $\Gamma$ be a group with the weak Haagerup property with constant 1. Then there exist a proper function $\varphi:\Gamma\to\mathbb{R}$, a Hilbert space $\mathcal{H}$, and maps $R,S:\Gamma\to\mathcal{H}$ such that
\begin{align*}
\varphi(y^{-1}x)=\lVert R(x)-R(y)\rVert ^2+\lVert S(x)+S(y)\rVert ^2,\quad\forall x,y\in\Gamma.
\end{align*}
In particular, $\varphi(e)=4\lVert S(x)\rVert ^2$ for all $x\in\Gamma$.
\end{thm}

This theorem allows us to prove the following result, which implies Theorem \ref{Thm_WA}.

\begin{thm}
Let $\Gamma$ be a group with the weak Haagerup property with constant 1. Then $\Gamma$ has a proper uniformly Lipschitz affine action on a subspace of an $L^1$ space.
\end{thm}
\begin{proof}
Let $R, S:\Gamma\to\mathcal{H}$ be the maps given by Theorem \ref{Thm_Knu}, and define a CND kernel $\psi:\Gamma\times\Gamma\to[0,\infty)$ by
\begin{align*}
\psi(x,y)=\lVert R(x)-R(y)\rVert ^2,\quad\forall x,y\in\Gamma.
\end{align*}
For all $s,x,y\in\Gamma$,
\begin{align*}
\psi(sx,sy)&=\lVert R(sx)-R(sy)\rVert ^2\\
&=\varphi(y^{-1}x)-\lVert S(sx)+S(sy)\rVert ^2\\
&=\lVert R(x)-R(y)\rVert ^2 + \lVert S(x)+S(y)\rVert ^2 - \lVert S(sx)+S(sy)\rVert ^2\\
&\leq \psi(x,y) + \varphi(e),
\end{align*}
where $\varphi$ is the function given by Theorem \ref{Thm_Knu}. Moreover,
\begin{align*}
\psi(x,e)= \varphi(x)-\lVert S(x)+S(e)\rVert ^2 \geq \varphi(x)-\varphi(e),
\end{align*}
which tends to infinity as $x\to\infty$ because $\varphi$ is proper. We conclude that $\psi$ is a proper almost invariant CND kernel. By Theorem \ref{Thm_caract}, $\Gamma$ has a proper uniformly Lipschitz affine action on a subspace of an $L^1$ space.
\end{proof}

\section{{\bf A-TTT-menable groups}}\label{S_a-TTT}

Recall that a countable group $\Gamma$ is a-TTT-menable if it admits a proper wq-cocycle. A map $b:\Gamma\to\mathcal{H}$ is called a wq-cocycle if there exists $\pi:\Gamma\to\mathcal{U}(\mathcal{H})$ such that
\begin{align*}
\sup_{s,t\in\Gamma}\lVert b(st)-\pi(s)b(t)-b(s)\rVert < \infty.
\end{align*}
The reason behind this terminology lies in the fact that a-TTT-menability is a strong negation of Property (TTT), as defined \cite{Oza2}; in the same way as a-T-menability is a strong negation of Property (T).

The proof of Theorem \ref{Thm_a-TTT} consists essentially in noticing that wq-cocycles give rise to almost invariant CND kernels.

\begin{proof}[Proof of Theorem \ref{Thm_a-TTT}]
Let $\Gamma$ be an a-TTT-menable group, and let $b:\Gamma\to\mathcal{H}$ be a proper wq-cocycle associated to $\pi:\Gamma\to\mathcal{U}(\mathcal{H})$. Then there exists a constant $C>0$ such that
\begin{align*}
\lVert b(st)-\pi(s)b(t)-b(s)\rVert \leq C,\quad\forall s,t\in\Gamma.
\end{align*}
Let us define $\psi:\Gamma\times\Gamma\to[0,\infty)$ by
\begin{align*}
\psi(x,y)=\lVert b(x)-b(y)\rVert,\quad\forall x,y\in\Gamma.
\end{align*}
We claim that $\psi$ is a proper almost invariant CND kernel. Indeed, it is CND because $\psi^2$ is CND; see \cite[Lemma 6]{Now3}. Now, for all $s,x,y\in\Gamma$,
\begin{align*}
\psi(sx,sy)&=\lVert b(sx)-b(sy)\rVert\\
&\leq \lVert b(sx)-\pi(s)b(x)-b(s)\rVert + \lVert \pi(s)b(x)-\pi(s)b(y)\rVert\\
&\quad + \lVert b(s)+\pi(s)b(y)-b(sy)\rVert\\
&\leq 2C + \lVert b(x)-b(y)\rVert\\
&= \psi(x,y) + 2C.
\end{align*}
Finally, $\psi(x,e)$ tends to infinity as $x\to\infty$ because $b$ is proper. By Theorem \ref{Thm_caract}, we conclude that $\Gamma$ has a proper uniformly Lipschitz affine action on a subspace of an $L^1$ space.
\end{proof}

As pointed out in \cite{Oza2}, hyperbolic groups are a-TTT-menable. Hence, by Theorem \ref{Thm_a-TTT}, every hyperbolic group has a proper uniformly Lipschitz affine action on a subspace of an $L^1$ space. This is essentially what was done in \cite[Theorem 1.2]{Ver}, although a-TTT-menability is not explicitly mentioned in that paper.

\bibliographystyle{plain} 

\bibliography{Bibliography}

\end{document}